\documentclass[12pt]{amsart}

\newtheorem{thm}{Theorem}
\newtheorem{cor}[thm]{Corollary}
\newtheorem{prop}[thm]{Proposition}
\newtheorem{lem}[thm]{Lemma}

\newtheorem{conjecture}[thm]{Conjecture}

\theoremstyle{definition}

\numberwithin{equation}{section}
\numberwithin{thm}{section}

\newcommand{\C}{\ensuremath{{\mathbb{C}}}}
\newcommand{\Z}{\ensuremath{{\mathbb{Z}}}}
\renewcommand{\P}{\ensuremath{{\mathbb{P}}}}
\newcommand{\Q}{\ensuremath{{\mathbb{Q}}}}

\newcommand{\p}{\ensuremath{{\mathfrak{p}}}}
\newcommand{\q}{\ensuremath{{\mathfrak{q}}}}
\newcommand{\fp}{\ensuremath{{\mathfrak{p}}}}
\newcommand{\fq}{\ensuremath{{\mathfrak{q}}}}
\newcommand{\fm}{\ensuremath{{\mathfrak{m}}}}
\newcommand{\fr}{\ensuremath{{\mathfrak{r}}}}
\newcommand{\cZ}{\mathcal Z}
\newcommand{\cY}{\mathcal Y}
\newcommand{\cX}{\mathcal X}

\renewcommand{\o}{\ensuremath{{\mathfrak{o}}}}

\DeclareMathOperator{\Gal}{Gal}

\DeclareMathOperator{\rad}{rad}
\DeclareMathOperator{\Aut}{Aut}

\begin{document}


\title[$ABC$ Implies a Zsigmondy Principle for Ramification]{$ABC$ Implies a Zsigmondy Principle for Ramification}

\author[A. Bridy and T. J. Tucker]{Andrew Bridy and Thomas J. Tucker}
\address{Andrew Bridy\\Department of Mathematics\\ University of Rochester\\
Rochester, NY, 14620, USA}
\email{abridy@ur.rochester.edu}
\address{Thomas J. Tucker\\Department of Mathematics\\ University of Rochester\\
Rochester, NY, 14620, USA}
\email{thomas.tucker@rochester.edu}

\date{}

\begin{abstract}
Let $K$ be a number field or a function field of characteristic 0. If $K$ is a number field, assume the $abc$-conjecture for $K$. We prove a variant of Zsigmondy's theorem for ramified primes in preimage fields of rational functions in $K(x)$ that are not postcritically finite. For example, suppose $K$ is a number field and $f\in K[x]$ is not postcritically finite, and let $K_n$ be the field generated by the $n$th iterated preimages under $f$ of $\beta\in K$. We show that for all large $n$, there is a prime of $K$ that ramifies in $K_n$ and does not ramify in $K_m$ for any $m<n$.
\end{abstract}

\keywords{Arithmetic Dynamics, Ramification, abc Conjecture}
\subjclass[2010]{37P05; 11G50, 14G25}
\thanks{The second author was partially supported by NSF Grant DMS-1501515.}

\maketitle

\section{Introduction}
Let $K$ be either a number field or a function field of characteristic
0 of transcendence degree 1 over its field of constants. Let $\phi\in K(x)$ be a rational function. Recall that the
morphism $\phi:\P^1(K)\to\P^1(K)$ is \emph{postcritically finite} if
the forward orbit of the ramification locus of $\phi$ is a finite
set. Let $\phi$ be a non-postcritically finite rational function of degree $d\geq 2$ and let $\beta\in\P^1(K)$. As is usual in dynamics, we use $\phi^n$ to denote the map $\phi$ composed with itself $n$ times. For each $n \geq 1$, let
$$K_n = K(\phi^{-n}(\beta))=K(\gamma\in \overline{K}:\phi^n(\gamma)=\beta).$$
It is a theorem of the first author and coauthors \cite{BIJJ} that for
any $\beta\in\P^1(K)$, there are infinitely many primes in $K$ that
ramify in $\bigcup_{n=1}^\infty K_n$.  The main idea of the
theorem is to produce prime divisors of $\phi^n(\alpha) - \beta$ for
$\alpha$ a critical point of $\phi$ with canonical height
$h_\phi(\alpha) > 0$.  The fact that there are infinitely many such
primes follows from \cite{SilInt}.  Various authors (see
\cite{Zsigmondy, Elkies, Rice, Krieger, IngramSilverman, FaberGranville, GNT} for
example) have sought to show that not only are there infinitely many
primes that divide $\phi^n(\alpha) - \beta$ for some $n$, but the stronger statement that there exists
an $N$ such that for all $n > N$, there is a prime that divides
$\phi^n(\alpha) - \beta$ that does not divide $\phi^m(\alpha) -
\beta$ for any $m < n$.  If this is true, one might say that there are
infinitely many primes dividing $\phi^n(\alpha) - \beta$ for some $n$
because after a certain point each ``new iterate'' $\phi^n(\alpha) -
\beta$ gives a ``new prime'' dividing $\phi^n(\alpha) - \beta$.  This is
sometimes referred as the ``Zsigmondy principle'', after Zsigmondy \cite{Zsigmondy}
who studied these questions in the context of primitive divisors of
$a^n - b^n$.

In this paper, we prove a Zsigmondy principle for ramification for
certain types of rational functions, including polynomials.  Our
results are conditional on the $abc$ conjecture when $K$ is a number field.  For
polynomials, our result is the following. Recall that if $K$ is a function field with field of constants $k$, $f$ is said to be {\em isotrivial} if there is an element $\sigma \in \overline{K}(x)$ of degree one such that $\sigma  \circ \phi \circ \sigma^{-1} \in \bar{k}(x)$.

\begin{thm}\label{cor: polynomials}
  Let $K$ be a number field or a function field of characteristic 0.
  Let $f\in K[x]$ be a polynomial with $\deg f\geq 2$ that is
  not postcritically finite.  If $K$ is a number field, assume the
  $abc$ conjecture for $K$.  If $K$ is a function field, assume that
  $f$ is not isotrivial. Then, for all sufficiently large $n$, there
  exists a prime of $K$ that ramifies in $K(f^{-n}(\beta))$ and does not
  ramify in $K(f^{-m}(\beta))$ for $m<n$.
\end{thm}

Our most general theorem is most easily stated in terms of grand orbits. The \emph{orbit} or \emph{forward orbit} of $\beta\in \P^1(K)$ is 
$$\mathcal{O}_\phi(\beta)=\{\phi^n(\beta):n\geq 0\}=\{\beta,\phi(\beta),\phi^2(\beta),\dots\}.$$
The \emph{backward orbit} of $\beta$ is
$$\mathcal{O}_\phi^-(\beta)=\{\alpha\in \P^1(\overline{K}):\phi^n(\alpha)=\beta\text{ for some }n\geq 0\}=\bigcup_{n=0}^\infty \phi^{-n}(\beta).$$
The \emph{grand orbit} of $\beta$ is the backward orbit of the forward orbit, that is, 
$$\mathcal{GO}_\phi(\beta)=\{\alpha\in \P^1(\overline{K}):  \phi^m(\alpha)=\phi^n(\beta)\text{ for some }m,n\in\Z_{\geq 0}\}.$$
Grand orbits under $\phi$ partition $\P^1(\overline{K})$ into equivalence classes. A point $\beta$ is said to be \emph{exceptional} for $\phi$ if its grand orbit is a finite set. It is well known that if $\beta$ is exceptional for $\phi$, then (up to conjugacy by a fractional linear transformation) either $\phi$ is a polynomial and $\beta=\infty$, or $\phi(x)=x^d$ for some $d\in\Z$ and $\beta\in\{0,\infty\}$.

A point $\beta\in\P^1(K)$ is \emph{periodic} if $\phi^n(\beta)=\beta$ for some $n>0$ and \emph{preperiodic} if $\phi^n(\beta)=\phi^m(\beta)$ for some $n>m\geq 0$. A point that is not preperiodic is \emph{wandering}. We define a grand orbit to be preperiodic if one (equivalently any) of its points is preperiodic, and wandering otherwise.

We now state the main theorem. If $K$ is a number field, we will assume that the $abc$ conjecture holds for $K$. If $K$ is a function field of characteristic 0, the $abc$ conjecture is a theorem of Mason-Stothers \cite{Mason,Stothers} (see also Silverman \cite{SilvermanABC}). As we now consider rational maps from $\P^1(K)$ to itself, it is possible for $\infty$ to arise as a preimage of $K$, in which case we simply declare that $K(\infty)=K$.

\begin{thm}\label{thm: main}
  Let $\phi\in K(x)$ with $\deg\phi\geq 2$. Suppose that $\phi$ is not postcritically finite and that $\beta \in \P^1(K)$ is not exceptional for $\phi$. If $K$ is a number field,
  assume the $abc$ conjecture for $K$. If $K$ is a function field,
  assume that $\phi$ is not isotrivial. Suppose that the ramification locus $R_\phi$ intersects
  at most $d-1$ distinct wandering grand orbits. For all sufficiently
  large $n$, there exists a prime of $K$ that ramifies in
  $K(\phi^{-n}(\beta))$ and does not ramify in $K(\phi^{-m}(\beta))$ for $m<n$.
\end{thm}

The restriction that $\phi$ be non-isotrivial is not a serious one.  Indeed, we can treat the case of isotrivial rational
functions by a fairly elementary argument, provided that $\beta$ is
not in the field of constants of $K$.  See Theorem \ref{iso}.  

Theorem \ref{thm: main} immediately produces Theorem \ref{cor:
  polynomials} as a special case, since a polynomial of degree $d$ has
at most $d-1$ critical points other than the point at infinity (which
is of course a fixed point).  For rational functions in general we have the following theorem, which shows that a new prime ramifies at every two levels in the tower of fields $K_n$. 

\begin{thm}\label{thm: every two levels}
Let $\phi\in K(x)$ with $\deg\phi\geq 2$. Suppose that $\phi$ is not postcritically finite and that $\beta \in \P^1(K)$ is not exceptional for $\phi$.  If $K$ is a number field, assume the $abc$
  conjecture for $K$.  If $K$ is a function field, assume that $\phi$ is
  not isotrivial.  For all sufficiently large $n$, there exists a
  prime of $K$ that ramifies in $K(\phi^{-n}(\beta))$  and does not ramify in
  $K(\phi^{-m}(\beta))$  for $m\leq n-2$.
\end{thm}

One of our motivations for proving Theorem \ref{thm: main} was an
application to the growth rate of Galois groups of iterates of
polynomials. The group $\Gal(K_n/K)$ injects into $\Aut(T_n)$, the
automorphism group of the complete $d$-ary rooted tree of height $n$ where
$d=\deg\phi$. The group $\Aut(T_n)$ is isomorphic to an iterated
wreath product of the symmetric group $S_d$, so $|\Aut(T_n)|$ grows
doubly exponentially in $n$. It is expected that in many cases the
index $|\Aut(T_n):\Gal(K_n/K)|$ remains bounded as $n\to\infty$, which
implies that the degree of the splitting field of $\phi^n(x)-\beta$
over $K$ grows doubly exponentially for large $n$. Odoni proved that
generic polynomials have this property, as well as the particular
polynomial $x^2-x+1$ \cite{OdoniIterates,OdoniWreathProducts}; Juul
\cite{Juul} proved that generic rational functions have this
property. Stoll proved that an infinite family of quadratic polynomials
\cite{Stoll} have this property.  Boston and Jones \cite{BJ} have
proposed a dynamical analog of the Serre open image theorem (see
\cite{Serre}), and we hope to use the techniques of this paper to
treat some special cases of this problem, in particular the case of
cubic polynomials.

It follows from our main theorem that the growth rate for many non-postcritically finite rational maps is at least simply exponential (conditional on the $abc$ conjecture when $K$ is a number field). For example, this includes all polynomial maps.
\begin{cor}\label{cor: Galois growth}
Suppose that $K$, $\phi\in K(x)$, and $\beta\in\P^1(K)$ satisfy the assumptions of Theorem \ref{thm: main}. Then there exists $C$ such that for all sufficiently large $n$, $[K(\phi^{-n}(\beta)):K]\geq C 2^n$.
\end{cor}

The strategy of our proof combines the approaches of \cite{GNT} and
\cite{BIJJ}.  We begin with Lemma \ref{necessary}, which gives a necessary condition
for $K_n$ to ramify over $\p$; this is
adapted from \cite{PeriodsModPrimes}.  We then prove  Lemma
\ref{sufficient}, which gives a sufficient condition for a prime $\p$
to ramify in $K_n$.  Note that the condition in both Lemmas has do
to with whether or not a suitable iterate of a critical point under
$\phi$ meets $\beta$ at $\fp$.  We then use a so-called ``Roth-abc'' result (see Proposition
\ref{Roth-abc}) to show that for each critical point $\alpha$ the
quantities $\phi^n(\alpha) - \beta_j$ have very few repeated factors
for large $n$ and suitable preimages $\beta_j$ of $\beta$.  This is
done in Lemma \ref{from-Roth}.  We are also able to bound the
contribution to $h(\phi^n(\alpha) - \beta_j)$ coming from primes that
divide $\phi^m(\alpha') - \beta_j$ for some $m < n$ and $\alpha'$
some critical point of $\phi$.  This is done in Lemmas \ref{from-5.1}
and \ref{other G} (note that in the application of Lemma \ref{other
  G}, it is crucial that the number of wandering grand orbits of
$\phi$ is small).  Putting these together along with some other simple
estimates gives a prime $\p$ such that
$v_\fp( \phi^n(\alpha) - \beta_j) = 1$ for some suitable $\beta_j$ of
$\beta$ with the property that $\p$ does not ramify in $K_m$
for any $m < n$.  Applying Lemma \ref{sufficient} then gives our main
result, Theorem \ref{thm: main}.
 
\vskip5mm

\noindent {\bf Acknowledgments.} We would like to thank the referee for
many useful suggestions and corrections.

\section{Preliminaries}\label{background}

Let $K$ be a number field or a function field of characteristic 0 with
transcendence degree 1 over its field of constants $k$. Let $\phi\in K(x)$ be a rational function of degree $d\geq 2$. If $K$ is a number field, let $\o_K$ be the ring of integers of $K$. If $K$ is a function field, choose a prime $\q$ and let $\o_K=\{z\in K:v_\p(z)\geq 0\text{ for all primes }\p\neq\q\text{ of }K\}$. For any prime $\p$, let $k_\p$ be the residue field $\o_K/\p$.

We use the notion of good reduction as introduced by Morton and
Silverman \cite{MortonSilverman}. Let $\phi:\P^1(K)\to \P^1(K)$ be a
morphism, written in homogeneous coordinates as
$\phi([X:Y])=[P(X,Y):Q(X,Y)]$, where $P,Q\in\o_K[X,Y]$ are homogeneous
polynomials of the same degree without any common factor in
$\overline{K}[X,Y]$.  Letting $P_0(X,Y) = P(X,Y)$ and $Q_0(X,Y) = Q(X,Y)$, we
recursively define $P_{m+1} = P( P_m(X,Y), Q_m(X,Y))$ and
$Q_{m+1} = Q(P_m(X,Y), Q_m(X,Y))$.  We let $p_m(X) = P_m(X,1)$ and
$q_m(X) = Q_m(X,1)$.
 
Let $\phi_\p=[P_\p:Q_\p]$, where $P_\p,Q_\p\in k_\p[X,Y]$ are the
reductions of $P$ and $Q$ modulo $\p$. We say that $\phi$ has
\emph{good reduction} at $\p$ when
$\max(\deg P_\p, \deg Q_\fp)$ equals  $\max (\deg P,\deg Q)$ and $P_\fp$ and
$Q_\fp$ have no common factor in $\bar{k}_\fp[X,Y]$.  When this is the
case, $\phi_\fp$ induces a nonconstant morphism from $\P^1_{k_\fp}$ to
itself.  When this map is separable, we say that $\phi$ has {\em good
  separable} reduction at $\p$.

\subsection{Heights}\label{heights}
For a rational prime $\p$ of $K$, define 
\begin{equation*}
N_\p=  \frac{1}{[K:\Q]} \log\#k_\p
\end{equation*}
if $K$ is a number field and
\begin{equation*}
N_\p=  [k_\p:k]
\end{equation*}
if $K$ is a function field. As in \cite{GNT}, normalizing by the degree of the number field will make it easier to state proofs in the same way for both number fields and function fields.

If $K$ is a number field, the height of $z\in K$ is defined as
\begin{equation*}
h(z)=-\sum_{\text{primes }\p\text{ of }\o_K} \min(v_\p(z),0)N_\p + \frac{1}{[K:\Q]}\sum_{\sigma:K\hookrightarrow\C}\max(\log|\sigma(z)|,0)
\end{equation*}
where the second sum is taken over all maps $\sigma:K\to \C$ (in particular, complex conjugate embeddings are not identified). We extend $h$ to $\P^1(K)$ by setting $h(\infty)=0$.
If $K$ is a function field, instead the height of $z\in K$ is
\begin{equation*}
h(z)=-\sum_{\text{primes }\p\text{ of }\o_K} \min(v_\p(z),0)N_\p.
\end{equation*}
In either case, for $z\neq 0$ the product formula gives the inequality
\begin{equation*}
\sum_{v_\p(z)>0}v_\p(z)N_\p\leq h(z).
\end{equation*}

We will use the Call-Silverman canonical height $h_\phi$, which is defined by
\begin{equation*}
h_\phi(x)=\lim_{n\to\infty}\frac{h(\phi^n(x))}{d^n}.
\end{equation*}
This limit exists by the same telescoping series argument that shows the existence of the Ner\'on-Tate height on an elliptic curve. See \cite{CallSilverman} for details. The canonical height satisfies the following important properties for some absolute constant $C_\phi$ and for every $x\in K$:
\begin{align*}
h_\phi(\phi(x)) & =dh_\phi(x)\text{, and}\\
|h(x)-h_\phi(x)| & \leq C_\phi.
\end{align*}
It follows immediately from these properties that $h_\phi(x)\neq 0$ if and only if $h(\phi^n(x))\to\infty$ as $n\to\infty$. 

If $K$ is a number field, then for $n\geq 2$ we define the height of the nonzero $n$-tuple $(z_1,z_2,\dots,z_n)\in K^n$ by
\begin{align*}
h(z)= & -\sum_{\text{primes }\p\text{ of }\o_K} \min(v_\p(z_1),\dots,v_\p(z_n))N_\p\\
 & + \frac{1}{[K:\Q]}\sum_{\sigma:K\hookrightarrow\C}\max(\log|\sigma(z_1)|,\dots,\log|\sigma(z_n)|)
\end{align*}

\subsection{The $abc$-conjecture} \label{abc}
For $z_1,\dots,z_n\in K^\times$, we define
\begin{equation*}
I(z_1,\dots,z_n)=\{\text{primes }\p\text{ of }\o_K\mid v_\p(z_i)\neq v_\p(z_j)\text{ for some }i,j \}
\end{equation*}
and
\begin{equation*}
\rad(z_1,\dots,z_n)=\sum_{\p\in I(z_1,\dots,z_n)} N_\p.
\end{equation*}
With this notation, we assume the $abc$-conjecture as follows.

\begin{conjecture}
Let $K$ be a number field. For any $\epsilon>0$, there exists a constant $C_{K,\epsilon}$ such that for all $a,b,c\in K^\times$ with $a+b=c$, we have $$h(a,b,c)<(1+\epsilon)\rad(a,b,c)+C_{K,\epsilon}.$$
\end{conjecture}

We will make use of the following estimate, sometimes called
``Roth-$abc$" as in \cite{GNT}, which holds for number fields
conditionally on the $abc$-conjecture and is true unconditionally for
function fields of characteristic 0.  The following combines
Propositions 3.4 and 4.2 from \cite{GNT}.
\begin{prop}\label{Roth-abc}
Let $K$ be a number field or function field of characteristic 0.  If
$K$ is a number field, suppose that the $abc$-conjecture holds for
$K$.  
Let $F\in K[x]$ be a polynomial of degree at least 3 with no repeated factors and let $\epsilon>0$. Then there exists $C_{F,\epsilon}$ such that for all $x\in K$,
$$\sum_{v_\p(F(x))>0}N_\p\geq (\deg F-2-\epsilon)h(x)+C_{F,\epsilon}.$$
\end{prop}

Note that in the case where $K$ is a function field, the result does
not follow from $abc$ but instead requires Yamanoi's proof
\cite{Yamanoi} of the
Vojta conjecture for algebraic points on curves over function fields
of characteristic 0.

\subsection{Base extension}

Certain arguments are made more easily after passing from our number
field or function field $K$ to a finite extension $L$ of $K$.  We will
quickly show that our results are true over $K$ exactly when they are
true over a finite extension.

\begin{lem}\label{base1}
  Let $K$ be a number field or function field of characteristic 0, let
  $L$ be a finite extension of $K$, let $\fp$ be a finite prime of $K$
  that does not ramify in $L$, and let $\fq$ be a finite prime of $L$
  such that $\fq | \fp$. Then, for any finite Galois extension $M$ of
  $K$, the prime $\fp$ ramifies in $M$ if and only if $\fq$ ramifies
  in the compositum $M \cdot L$.
\end{lem}
\begin{proof}
Suppose that $\fp$ does not ramify in $M$.  Then $\fp$ does not ramify
in $M \cdot L$ since $\fp$ does not ramify in $L$.  Thus, any prime
$\fq$ of $L$ such that $\fq | \fp$ cannot ramify in $M \cdot L$.

Suppose that $\fp$ ramifies in $M$.  Since $M$ is Galois over $K$,
this means that $e(\fm/\fp)  > 1$ for any $\fm | \fp$ in $M$.  Thus,
for any $\fr | \fp$ in $L \cdot M$, we have $e(\fr/\fp) > 1$.  Since
$e(\fq/\fp) = 1$, we must have
$e(\fr/\fp) = e(\fr/\fq)$; hence, 
$e(\fr / \fq) > 1$ so $\fq$ ramifies in  $M \cdot L$.  
\end{proof}

\begin{lem}\label{base2}
Let $K$ be a number field or function field of characteristic 0, let
$\beta \in K$, and let $\phi$ be a rational function with coefficients
in $K$.  Let $L$ be a finite extension of $K$.
Then the following statements are equivalent:
\begin{itemize}
\item[(a)]  For all sufficiently large $n$, there is a finite prime
  $\fp$ of $K$ such that $\fp$ ramifies in $K(\phi^{-n}(\beta))$ and
  $\fp$ does not ramify in $K(\phi^{-m}(\beta))$ for $m < n$.
\item [(b)] For all sufficiently large $n$, there is a finite prime
  $\fq$ of $L$ such that $\fq$ ramifies in $L(\phi^{-n}(\beta))$ and
  $\fq$ does not ramify in $L(\phi^{-m}(\beta))$ for $m < n$.  
\end{itemize}
\end{lem}
\begin{proof}
Let $S$ be the set of finite primes of $K$ that ramify in $L$ and let $T$ be
the set of primes of $L$ that lie over primes in $S$.  

Suppose that (a) holds.  Then, since $S$ is finite, for all
sufficiently large $n$, there is a finite prime $\fp \notin S$ of $K$
such that $\fp$ ramifies in $K(\phi^{-n}(\beta))$ and $\fp$ does not
ramify in $K(\phi^{-m}(\beta))$ for $m < n$.  If $\fq$ is a prime of
$L$ such that $\fq | \fp$, then $\fq$ ramifies in
$L(\phi^{-n}(\beta))$ and $\fq$ does not ramify in
$L(\phi^{-m}(\beta))$ for $m < n$, by Lemma \ref{base1}.

Likewise, if (b) holds, then, since $T$ is finite, for all
sufficiently large $n$, there is a finite prime $\fq \notin T$ of $L$
such that $\fq$ ramifies in $L(\phi^{-n}(\beta))$ and and $\fq$ does
not ramify in $L(\phi^{-m}(\beta))$ for $m < n$.  If $\fp$ is a prime of $K$ 
such that $\fq | \fp$, then  $\fp$ ramifies in $K(\phi^{-n}(\beta))$ and
  $\fp$ does not ramify in $K(\phi^{-m}(\beta))$ for $m < n$, again by Lemma
  \ref{base1}.  
\end{proof}

By Lemma~\ref{base2}, it suffices to prove Theorem \ref{thm: main}
over a finite extension $L$ of $K$. We argue here that it also suffices to prove the Theorem after replacing 
$\phi$ with $\phi^\sigma=\sigma\circ\phi\circ\sigma^{-1}$ for 
any M\"obius transformation $\sigma\in L(x)$, and replacing $\beta$ with 
$\sigma(\beta)$. Note that for any $\phi\in K(x)$ and $\beta\in\P^1(K)$, 
the hypotheses of Theorem \ref{thm: main} ($\phi$ is postcritically 
finite, $\beta$ is non-exceptional, and the condition on wandering grand orbits 
intersecting $R_\phi$) are invariant under this change of variables. 
This is because $\alpha$ is a critical point of $\phi$ 
if and only if $\sigma(\alpha)$ is a critical point of $\phi^\sigma$, and because the map $\sigma$ 
induces a bijection from the grand orbits of $\phi$ to the grand orbits of $\phi^\sigma$ 
that preserves their structure as grand orbits. Thus, we may assume that $\phi$ has a
fixed point defined over $K$, and, after changing variables, we may
assume that $\phi(\infty)=\infty$. Note that this means that 
$\deg P_m > \deg Q_m$ for all $m$ and that when
$\phi$ has good reduction at $\p$, the leading coefficient of $P_m$ is
not divisible by $\p$ for all $m$.

\section{Criteria for ramification}

To prove Theorem 1.1, we need some conditions for ramification in preimage fields. The necessary condition is an adaptation of a standard result about ramification in $\p$-adic fields, for example \cite[Lemma 1]{PeriodsModPrimes}. Recall that $K$ is either a number field or a function field of characteristic 0. From this point forward, for $\phi\in K(x)$ and $\beta\in\P^1(K)$, we use the notation $K_n=K(\phi^{-n}(\beta))$ as defined in the introduction.

\begin{prop}\label{necessary}
  Let $\phi\in K(x)$ and $\beta\in K$. Let $\p$ be a prime of $K$ such that $\phi$ has good separable
  reduction and $v_\fp(\beta) \geq 0$. If $\p$ ramifies in $K_n$, there exists
  $\alpha\in R_\phi$ such that $v_\p(\phi^m(\alpha)-\beta)>0$ for some
  $m$ with $1\leq m\leq n$.
\end{prop}
\begin{proof}
  Let $(p_n)_\fp$ and $(q_n)_\fp$ denote the reductions of $p_n$ and
  $q_n$ at $\p$, and let $\beta_\fp$ denote the reduction of $\beta$
  at $\p$.  Since $K_n$ is the splitting field of
  $p_n(X) - \beta q_n(X)$, it follows that if $K_n$ ramifies at
  $\p$ then $F(X)=(p_n)_\p(X) - \beta_\p (q_n)_\p(X)$ has a multiple root.
  Thus, there is a root of $F(X)$ that is also a root of the derivative of $F(X)$.   
  
Note that if $\gamma$ is a root of both $F(X)$ and $F'(X)$, then $\gamma$ is also a root of $(p_n)_\p'(X) (q_n)_\p(X) - (p_n)_\p(X) (q_n)_\p'(X)$.  Since $(\phi_\fp)^n$ is separable at $\fp$,
we see that $(p_n)_\p'(X) (q_n)_\p(X) - (p_n)_\p(X) (q_n)_\p'(X)$ is not identically zero.  Hence, all of its roots are the reduction modulo $\p$ of
a root of $p_n'(X) q_n(X) - p_n(X) q_n'(X)$.  Therefore, there is a
critical point $\alpha$ of $\phi^n$ that reduces to a root of $(p_n)_\p(X)
- \beta (q_n)_\p(X)$ at $\fp$.  This means that
$v_\p(\phi^m(\alpha)-\beta)>0$.      
\end{proof}

\begin{prop}\label{sufficient}
Let $\phi\in K(x)$ and $\beta\in K$. For all primes $\p$ of $K$ such that $\phi$ has good separable reduction at $\fp$ and $v_\fp(\beta) \geq 0$, if there exists a critical point $\alpha$ of $\phi$ such that $\phi^n(\alpha)\neq\infty$ and $v_\p(\phi^n(\alpha)-\beta)=1$, then $\p$ ramifies in $K_n$.
\end{prop}
\begin{proof}
This is the criterion that forms the main argument of \cite[Theorem
5]{BIJJ}.  We provide a brief proof here.  First note that by Lemma \ref{base1}, we may 
assume without loss of generality that $\alpha\in K$, as otherwise we can
replace $K$ by $K(\alpha)$. 

Since $K_n$ is the
splitting field of $p_n(X) - q_n(X) \beta$ and $\alpha \in K$, it
follows that $K_n$ is also the splitting field of the
polynomial $p_n(X+\alpha) - q_n(X+\alpha) \beta$.   We write
\[p_n(X+\alpha) - q_n(X+\alpha) \beta = a_k X^k + \dots + a_0.\]
Note that $v_\fp(a_0) = v_\fp((\phi^n(\alpha) - \beta) q_n(\alpha)) = 1$,
because $v_\fp(q_n(\alpha)) = 0$ since $v_\fp(\beta) \geq 0$ and
$\phi^n$ has good reduction at $\p$.   Also note that $v_\fp(a_k) =
0$, again using the fact that $\phi^n$ has good reduction at $\p$.

Now, $p_n(X+\alpha) - q_n(X+\alpha) \beta$ is congruent mod $\p$ to
$p_n(X+\alpha) - q_n(X+\alpha) \phi^n(\alpha)$, because
$v_\p(\phi^n(\alpha)-\beta) > 0$.  We have that $X^e$ divides
$p_n(X+\alpha) - q_n(X+\alpha) \phi^n(\beta)$, where $e >1$ is the
ramification index of $\alpha$, so there is an $\ell > 1$ such that $v_\fp(a_j) > 0$ for $k= 0, \dots,
\ell - 1$ and $v_\fp(a_\ell) = 0$.  Thus, the first segment of the $\fp$-adic Newton
polygon of $p_n(X+\alpha) - q_n(X+\alpha) \beta$ is the line from
$(0,1)$ to $(\ell,0)$.  Therefore, $p_n(X+\alpha) - q_n(X+\alpha)$ has
a root $\gamma$ such that $v_\fp(\gamma)  = 1/\ell$, which means that
$K_n$ ramifies over $K$ at $\p$.  (See \cite[IV.3]{Koblitz} for
summary of the theory of Newton polygons.)

\end{proof}
In the next section, we will use Propositions \ref{necessary} and \ref{sufficient} in tandem to show the existence of primes that ramify in the $n$th preimage field but do not ramify earlier.

\section{Proofs of Main Theorems}

To prove Theorem \ref{thm: main}, we want to reduce to the case where
the base point $\beta$ is non-periodic and non-postcritical. This
ensures that the preimage sets $\phi^{-n}(\beta)$ are of size $d^n$,
and in particular, that the numerator of $\phi^n(x)-\beta$ is a
squarefree polynomial. This will allow us to easily use the Roth-$abc$
estimate of Proposition \ref{Roth-abc}. Of course, in general $\beta$
may be periodic or postcritical.  Let $t$ be the smallest positive
integer such that no element of 
$\phi^{-t}(\beta) \setminus \phi^{-(t-1)}(\beta)$ is periodic or
postcritical.  Let $\{ \beta_1, \dots, \beta_N\}$ denote
$\phi^{-t}(\beta) \setminus \phi^{-(t-1)}(\beta)$.  Note that that if
$x\in\phi^{-n}(\beta)$ for some $n>t$, and $x$ is not periodic, not
critical, and not postcritical, then
$x\in\bigcup_{j=1}^N\mathcal{O}_\phi^-(\beta_j)$. By the discussion at the end of 
Section~\ref{background}, we may adjoin the critical points of $\phi$ and the points 
$\beta_1,\dots,\beta_N$ to $K$, and also make a change of variables such that
$\phi(\infty)=\infty$.  

We construct a finite set of bad primes $S$ for which we may not be
able to control ramification. Let $S$ contain the primes $\p$ where
$\phi$ does not have good separable reduction at $\p$, $v_\p(\beta_j) \neq 0$ for some $\beta_j$,  $v_\p(\beta_j-\beta_k)>0$ for $j\neq k$, or $v_\fp(\phi^m(\gamma)-\beta_j)>0$ for some $m\in\{0,\dots,t-1\}$ and some $\gamma\in R_\phi$.  Theorem \ref{thm: main} will
be a straightforward consequence of the following three lemmas.

\begin{lem}\label{from-5.1}
Let $\alpha\in\P^1(K)$ with $h_\phi(\alpha)>0$ and let $\beta_1,\dots,\beta_N$ be as above. If $K$ is a number field, assume the $abc$ conjecture for $K$. Let $\delta> 0$.  For $n>0$, let $\cZ(n)$ denote the set of
primes $\p\notin S$ such that 
\[\min(v_\p(\phi^n(\alpha)-\beta_i),v_\p(\phi^m(\alpha)-\beta_j))>0\]
for some $0 < m < n$ and some $i, j$ between 1 and $N$.  Then there exists a constant $C_\delta$ such that
\begin{equation*}\label{5.1}
\sum_{\p\in\cZ(n)} N_\p\leq \delta d^n h_\phi(\alpha)+C_\delta
\end{equation*}
for all sufficiently large $n$.
\end{lem}

\begin{proof}

Let $F(X) = \prod_{i=1}^N(X- \beta_i)$.  Then $F$ divides the numerator of $\phi^t$ (because $\phi^t(\beta_i) =
0$ for all $i$), none of the $\beta_i$ are periodic, and
$\phi^\ell(\beta_i) \not= 0$ for all $i$ and any $\ell = 0, \dots
t-1$.  Then Proposition 5.1 of \cite{GNT} asserts that if $\cZ'(n)$ is
the set of primes $\fp$ such that $\min(v_\fp(\phi^{m+t}(\alpha)),
v_\fp(F(\phi^{n}(\alpha)))) > 0$, then for any $\delta > 0$, there
is a constant $C_\delta$ such that
\[ \sum_{\p\in\cZ'(n)} N_\p\leq \delta  h(\phi^n(\alpha))+C_\delta \] 
for all $n$.  If $\fp \notin S$  and 
\[ \min(v_\p(\phi^n(\alpha)-\beta_i),v_\p(\phi^m(\alpha)-\beta_j)) >
0, \]
then $v_\fp(\phi^{m+t}(\alpha)) > 0$, since $\phi^t(\beta_j) = 0$, and
$v_\fp(F(\phi^{n}(\alpha))) > 0$ since $\beta_i$ is a root of $F$.
Thus, we see that $\cZ(n) \subseteq \cZ'(n)$. Using the properties of
$h_\phi$ established in Section \ref{background}, namely that $h_\phi(\phi(x))=dh_\phi(x)$ and that
$|h(x)-h_\phi(x)|$ is bounded independently of $x$, our proof is complete.
\end{proof}

\begin{lem}\label{from-Roth}
Let $\beta_j$ be as above. If $K$ is a number field, suppose that the
$abc$-conjecture holds for $K$.  For every $\epsilon > 0$, there is a constant $C_\epsilon$ such that
\begin{equation*}\label{lower bound}
\sum_{v_\p(\phi^n(\alpha)-\beta_j)= 1}N_\p\geq
(d-\epsilon)d^{n-1}h_\phi(\alpha) + C_\epsilon.
\end{equation*}
\end{lem}
\begin{proof}
  Choose $m > 0$ such that $3/d^m < \epsilon/d$.  
  Since $\beta_j$ is not in the post-critical set, for any $m$,
  the set of solutions to $\phi^m(x) = \beta_j$ consists of exactly $d^m$
  distinct points.  Thus, $p_m(X) - \beta_j q_m(X)$ has no repeated roots.  Thus, using
  Proposition \ref{Roth-abc}, and the fact that $|h- h_\phi|$ is
  bounded, there is a constant $C_1$ such that
  \[ \sum_{v_\p(p_m(x) - \beta_j q_m(x)) = 1} N_\p\geq (d^m-3) h_\phi(x)
  +C_1 \]
for all $x \in K$.  
Letting $x = \phi^{n-m}(\alpha)$, we see there is a
constant $C_2$ such that
\[ \sum_{v_\p(\phi^m(\phi^{n-m}(\alpha))= 1}N_\p \geq  (1 -  \epsilon/d) d^{m}
d^{n-m} h_\phi(\alpha) +C_1 \geq (d-\epsilon)d^{n-1}h_\phi(\alpha)+ C_2.
 \] 
For all but at most finitely many $\fp$ we have
  $v_\p(\phi^n(\alpha)) = v_\fp(F(\phi^{n}(\alpha)))$, so the Lemma follows immediately.  
\end{proof}

\begin{lem}\label{other G}
Let $G$ be a set of critical points of $\phi$ that all have the same grand orbit. Let $\cY(i,j)$ be the set of primes $\p$ such that $$v_\p(\phi^i(\gamma)-\beta_j)>0$$ for some $\gamma\in G$. Let $M_G = \max_{\gamma \in G} h_\phi(\gamma)$.
Then, for all $n$, we have
\begin{equation*}\label{grand orbit}
\sum_{i=1}^{n-1}\sum_{j=1}^N\sum_{\p\in \cY(i,j)}N_\p\leq N \left(\frac{1}{d-1}\right)d^nM_G + O(n).
\end{equation*}
\end{lem}
\begin{proof}

Let $\alpha\in G$ be the critical point of largest canonical height $h_\phi(\alpha)$. For every $\gamma\in G$, we have $\phi^n(\alpha)=\phi^m(\gamma)$ for some $n,m\geq 0$, so $d^nh_\phi(\alpha)=d^mh_\phi(\gamma)$ and $m\geq n$. In other words, $\alpha$ is the ``farthest forward" critical point in the grand orbit. So except for $1\leq i\leq m-n$, the primes that divide $\phi^i(\gamma)-\beta_j$ also divide $\phi^k(\alpha)-\beta_j$ for some $k$. The indicated initial values of $i$ have a finite contribution to the sum that can be absorbed into the $O(n)$ term.

By the product formula and properties of heights we have
$$\sum_{v_\p(\phi^i(\alpha)-\beta_j)>0}N_\p\leq h(\phi^i(\alpha)-\beta_j)\leq d^ih(\alpha) + h(\beta_j)+C_\phi.$$
So we can use the estimation
\begin{align*}
\sum_{i=1}^{n-1}\sum_{j=1}^N\sum_{\p\in \cY(i,j)}N_\p\leq NM_G\frac{d^n-1}{d-1} + nC_{\phi,\beta_1,\dots,\beta_N}+O(n)
\end{align*}
and the lemma follows.
\end{proof}

Now, we are ready to prove Theorem \ref{thm: main}.
\begin{proof}[Proof of Theorem \ref{thm: main}.]  
Assume that $\phi\in K(x)$ is not postcritically finite and that
$\beta\in\P^1(K)$ is not exceptional for $\phi$. Let
$\beta_1,\dots,\beta_N$ be as above. If necessary, replace $K$ with $K(\alpha,\beta_1,\dots,\beta_N$) (by Lemma~\ref{base2} this loses no generality). Let $g$ be the number of
wandering grand orbits that $R_\phi$ intersects (we have $g\leq d-1$)
and let $\alpha\in R_\phi$ be a critical point of maximum canonical
height $h_\phi(\alpha)$. Observe that $h_\phi(\alpha)>0$, because if
every critical point has canonical height 0, then $\phi$ is
postcritically finite.   This follows from the fact that if $K$ is a
number field, then any nonpreperiodic point must have positive
canonical height by Northcott's theorem, while if $K$ is a function
field, Baker \cite{Baker1} and Benedetto \cite{Ben1} have proved that
any nonpreperiodic point has positive canonical height whenever $\phi$
is not isotrivial.  Hence, we may apply Lemma \ref{from-5.1} to the orbit of $\alpha$.

Let $\cX(n)$ be the set of primes $\p\notin S$ such that 
\begin{itemize}
\item $v_\p(\phi^n(\alpha)-\beta_j)=1$ for some $1\leq j\leq N$.
\item $v_\p(\phi^m(\alpha)-\beta_j)\leq 0$ for all $1\leq m\leq n-1$ and $1\leq j\leq N$, and 
\item $v_\p(\phi^m(\gamma)-\beta_j)\leq 0$ for every critical point $\gamma$ not in the same grand orbit as $\alpha$, and all $1\leq m\leq n-1$ and $1\leq j \leq N$.
\end{itemize}
For the critical points $\gamma$ in the same grand orbit as $\alpha$, we have $\phi^u(\gamma)=\phi^s(\alpha)$ for some positive integers $s,u$ with $s\leq u$ ($\alpha$ is the farthest forward critical point in its grand orbit as in the proof of Lemma \ref{other G}). So if $v_\p(\phi^m(\gamma)-\beta_j)> 0$ for some $m<n$, then either $m<u$ and $\fp\in S$, or $\phi^{s-u+m}(\alpha)\equiv\beta_j\pmod{\p}$, and $\fp\notin\cX(n)$ because $v_\p(\phi^m(\alpha)-\beta_j)\leq 0$ for $m<n$ if $\fp\in S$. Therefore if $\p\in\cX(n)$, then $v_\p(\phi^m(\gamma)-\beta_j)\leq 0$ for every critical point $\gamma$ of $\phi$ and every $m<n$. Thus by Propositions \ref{necessary} and \ref{sufficient}, $\p$ ramifies in $K_n$ and does not ramify in $K_m$ for $m<n$.

We show that $\cX(n)$ is nonempty for all large $n$. By Lemma \ref{from-Roth}, for a given $j$ and any $\epsilon>0$ we have
$$\sum_{v_\p(\phi^n(\alpha)-\beta_j)= 1}N_\p\geq(d-\epsilon)d^{n-1}h_\phi(\alpha) + C_\epsilon.$$
It follows that
$$\sum_{v_\p(\phi^n(\alpha)-\beta_j)= 1\text{ for some }j}N_\p\geq N(d-\epsilon)d^{n-1}h_\phi(\alpha) + C_\epsilon$$
because the primes $\p$ such that $v_\p(\phi
^n(\alpha)-\beta_j)>0$ for $j=j_1$ and $j=j_2$ are divisors of $\beta_{j_1}-\beta_{j_2}$, so these primes are contained in $S$, and their contribution to the sum can be absorbed into the constant $C_\epsilon$.

Now we apply Lemma \ref{from-5.1} to $\alpha$ and each $\beta_j$, and we apply
Lemma \ref{other G} to the grand orbits not containing $\alpha$ that intersect $R_\phi$. There are at most
$d-2$ such wandering grand orbits; any preperiodic grand orbits
contribute at most an $O(n)$ term to the sum because the term $M_G$
coming from Lemma \ref{other G} is zero. Now we subtract the conclusion of Lemma \ref{from-5.1} ($N$
times) and Lemma \ref{other G} ($g-1$ times) from the conclusion of Lemma \ref{lower bound}. This gives 
the following: for every $\epsilon>0$ and $\delta>0$, there
are constants $C_\epsilon$, $C_\delta$, and $C$ such that, for all
sufficiently large $n$, we have
\begin{align*}
\sum_{\p\in \cX(n)} N_\p \geq &
N(d-\epsilon)d^{n-1}h_\phi(\alpha) + C_\epsilon-N \delta d^n
                                h_\phi(\alpha) - C_\delta \\ & -  (g-1)N\frac{1}{d-1}d^nh_\phi(\alpha)+Cn\\
\geq & d^nh_\phi(\alpha)N\left(1-\epsilon d^{-1}-\delta-\frac{d-2}{d-1}\right)+Cn.
\end{align*}
Choosing $\epsilon$ and $\delta$ small enough, this quantity is positive for all large $n$, and we are done.

\end{proof}  

\begin{proof}[Proof of Theorem \ref{thm: every two levels}]
By the chain rule, the critical points of $\phi^2$ are either critical points of $\phi$ or preimages of these points under $\phi$, so the critical points of $\phi^2$ lie in at most $\# R_\phi\leq 2d-2$ distinct grand orbits. We have $2d-2<d^2-1$ because $d>1$. Applying Theorem \ref{thm: main} to the map $\phi^2$ and the point $\beta$, and also to a distinct point in $\phi^{-1}(\beta)$ (which exists because $\beta$ is not exceptional) yields the result. 
\end{proof}

\begin{proof}[Proof of Corollary \ref{cor: Galois growth}]
By Theorem \ref{cor: polynomials}, for all sufficiently large $n$ there is a prime of $K$ that ramifies in $K_{n+1}$ but not in $K_n$. Therefore the kernel of the natural surjection $\Gal(K_{n+1}/K)\to\Gal(K_n/K)$ is nontrivial, so it must be at least order 2. The result follows.
\end{proof}

\section{The isotrivial case}

In this section we treat the case of isotrivial rational functions.
The techniques here are much more elementary than in the rest of the paper.

\begin{thm}\label{iso}
  Let $K$ be a function field of characteristic 0 with field of
  constants $k$, and let $\phi \in K(x)$ be a rational function of
  degree greater than one.  Suppose that there is a finite extension
  $K'$ of $K$ and $\sigma \in K'(x)$ such that $\sigma \phi \sigma^{-1}
  \in k'(x)$, where $k'$ is the algebraic closure of $k$ in $K'$. Then
  we have the following:
\begin{itemize}
\item[(a)] If $\sigma(\beta) \in k'$, then there are at most finitely
  many primes of $K$ that ramify in $\bigcup_{n=1}^\infty K_n$.
\item[(b)] If $\sigma(\beta) \notin k'$ and $\phi$ is not
  postcritically finite, then for all sufficiently large $n$, there
  exists a prime of K that ramifies in $K_n$ and does not
  ramify in $K_m$ for $m < n$.
\end{itemize}
\end{thm}
\begin{proof}
Suppose that $\sigma(\beta) \in k'$.  Then, if $\phi^n(\alpha) =
\beta$, we have 
\[ \sigma \phi \sigma^{-1}(\sigma(\alpha)) = \sigma(\beta) \in k'. \]
Since $\sigma \phi \sigma^{-1} \in k'(x)$, where $k'$ is algebraic
over $k$, it follows that $\sigma(\alpha) \in \bar{k}$.  Thus, $\alpha$
is in the compositum $\bar{k} \cdot K'$.  Since $K'$ ramifies over 
at most finitely many primes of $K$ and $\bar{k} \cdot K'$ is
unramified everywhere over $K'$, we see that  $\bar{k} \cdot K'$
ramifies over at most finitely many primes of $K$.  Thus, there are
only finitely many primes of $K$ that ramify in $\bigcup_{n=1}^\infty
K_n$.

Now suppose that $\sigma(\beta) \notin k'$.  After passing to a finite
extension, we may assume that all the critical points of $\phi$ are
defined over $K'$.  Let $\phi^\sigma$ denote $\sigma \phi
\sigma^{-1}$.  Since every critical point of $\phi^\sigma$ is simply
$\sigma(z)$ for a critical point $z$ of $\sigma$ and every critical
point of $\phi^\sigma$ is algebraic over $k$, we see then that every
critical point of $\phi^\sigma$ is in $k'$.  

Now, note that $\sigma(\beta)$ is not algebraic over $k'$, and that
$K'$ is therefore a finite extension of $k'(\sigma(\beta))$.  For any
critical point $\alpha'$ of $\phi^\sigma$ and any $m$, we see that
$(\phi^\sigma)^m(\alpha') - \sigma(\beta)$ generates a prime in
$k'(\sigma(\beta))$.  Since $\phi^\sigma$ is not postcritically
finite, there is a critical point $\alpha$ of $\phi^\sigma$ such that
$(\phi^\sigma)^m(\alpha) \not= (\phi^\sigma)^n(\alpha')$ for any
$n < m$ and any critical point $\alpha \not= \alpha'$.  Thus, for
every $n >0$, there is a prime $\fm$ of $k'(\sigma(\beta))$ such that 
$v_\fm((\phi^\sigma)^n(\alpha) - \sigma(\beta)) = 1$ and
$v_\fm((\phi^\sigma)^m(\alpha)' - \sigma(\beta)) = 0$ for all $m < n$.
Then, by Proposition \ref{sufficient} and \ref{necessary}, this prime
$\fm$ ramifies in
$k'(\sigma(\beta)) ( (\phi^\sigma)^{-n} (\sigma(\beta)))$ and does
not ramify in
$k'(\sigma(\beta)) ((\phi^\sigma)^{-m} (\sigma(\beta)))$ for any
$m < n$.  Note that since $\sigma$ is defined over $K'$ and
$(\phi^\sigma)^n = \sigma \phi^n \sigma^{-1}$, we see that for any $z$
we have $(\phi^\sigma)^n(z) =\sigma(\beta)$ if and only if
$\phi^n(\sigma(z)) = \beta$.  Thus, by Lemma \ref{base2} it follows
that for all but finitely many $n$, there is a prime $\fq$ of $K'$ such
that $\fq$ ramifies in $L(\phi^{-n}(\beta))$ but $\fq$ does not ramify
in $L(\phi^{-m}(\beta))$ for any $m < n$.  Applying Lemma \ref{base2}
again, we see that for all but finitely many $n$, there is a prime
$\fp$ of $K$ such that $\fp$ ramifies in $K(\phi^{-n}(\beta))$ but
$\fp$ does not ramify in $K(\phi^{-m}(\beta))$ for any $m < n$, as
desired.

\end{proof}


\newcommand{\etalchar}[1]{$^{#1}$}
\providecommand{\bysame}{\leavevmode\hbox to3em{\hrulefill}\thinspace}
\providecommand{\MR}{\relax\ifhmode\unskip\space\fi MR }
\providecommand{\MRhref}[2]{%
  \href{http://www.ams.org/mathscinet-getitem?mr=#1}{#2}
}
\providecommand{\href}[2]{#2}

\end{document}